\documentclass[11pt,a4paper,english,reqno]{amsart}
\usepackage{amsmath,amssymb,amsfonts,epsfig,mathrsfs}
\usepackage[T1]{fontenc}

\usepackage{color}
\usepackage{array}
\usepackage{amsthm}
\usepackage{amstext}
\usepackage{graphicx}
\usepackage{setspace}
\usepackage[margin=2.5cm]{geometry}
\usepackage{bbm}
\usepackage{color}
\usepackage{enumitem}
\usepackage{undertilde}
\allowdisplaybreaks[4]

\usepackage{pgfplots}

\usepackage{amscd,psfrag}
\usepackage{yhmath}
\usepackage[mathscr]{eucal}
\usepackage{slashed}

\makeatletter
\pdfpageheight\paperheight
\pdfpagewidth\paperwidth

\usepackage{mathrsfs}

\usepackage{epstopdf}
\usepackage{chngcntr}
\counterwithin{figure}{section}
\usepackage{mathrsfs}

\usepackage{indentfirst}	

\usepackage[normalem]{ulem}
\theoremstyle{plain}
\newtheorem{definition}{Definition}[section]
\newtheorem{theorem}[definition]{Theorem}
\newtheorem*{theorem*}{Theorem}

\newtheorem*{remark*}{Remark}
\newtheorem*{sideremark*}{Side Remark}\newtheorem*{mt*}{Main Theorem}

\newtheorem*{claim*}{Claim}
\newtheorem*{q*}{Question}
\newtheorem{lemma}[definition]{Lemma}

\newtheorem*{corollary*}{Corollary}
\newtheorem*{proposition*}{Proposition}

\newcommand{\R}{\mathbb{R}}
\newcommand{\C}{\mathbb{C}}
\newcommand{\na}{\nabla}

\newcommand{\dd}{{\rm d}}
\newcommand{\p}{\partial}
\newcommand{\e}{\epsilon}
\newcommand{\emb}{\hookrightarrow}

\newcommand{\map}{\rightarrow}
\newcommand{\G}{\Gamma}
\newcommand{\M}{\mathcal{M}}

\newcommand{\1}{\mathbbm{1}}
\newcommand{\F}{\mathcal{F}}
\newcommand{\two}{{\rm II}}
\newcommand{\n}{\mathbf{n}}
\newcommand{\weakstar}{\buildrel\ast\over\rightharpoonup}
\newcommand{\sd}{\mathbb{S}^d}
\newcommand{\te}{\Theta^\e}

\allowdisplaybreaks[4]

\def\Xint#1{\mathchoice
{\XXint\displaystyle\textstyle{#1}}%
{\XXint\textstyle\scriptstyle{#1}}%
{\XXint\scriptstyle\scriptscriptstyle{#1}}%
{\XXint\scriptscriptstyle\scriptscriptstyle{#1}}%
\!\int}
\def\XXint#1#2#3{{\setbox0=\hbox{$#1{#2#3}{\int}$ }
\vcenter{\hbox{$#2#3$ }}\kern-.6\wd0}}

\def\dashint{\Xint-}

\numberwithin{equation}{section}
\numberwithin{figure}{section}

\title{A Remark on the Non-Compactness of $W^{2,d}$-Immersions of $d$-Dimensional Hypersurfaces}

\author{Siran Li}
\address{Siran Li: Department of Mathematics, Rice University, MS 136
P.O. Box 1892, Houston, Texas, 77251-1892, USA \, $\bullet$ \,  Department of Mathematics, McGill University, Burnside Hall, 805 Sherbrooke Street West, Montreal, Quebec, H3A 0B9, Canada}

\email{\texttt{Siran.Li@rice.edu}}


\date{\today}
\begin{document}

\maketitle

\begin{abstract}
We consider the continuous $W^{2,d}$-immersions of $d$-dimensional hypersurfaces in $\R^{d+1}$ with second fundamental forms uniformly bounded in $L^d$. Two results are obtained: first, we construct a family of such immersions whose limit fails to be an immersion of a manifold. This addresses the endpoint cases in J. Langer \cite{la} and P. Breuning \cite{br}. Second, under the additional assumption that the Gauss map is slowly oscillating, we prove that any family of such immersions subsequentially converges to a set locally parametrised by H\"{o}lder functions.

\end{abstract}

\section{Introduction}
	%Throughout this paper, an immersed submanifold $\psi: \M \emb \R^n$ is equipped with the pullback metric from the Euclidean metric on $\R^n$. We denote by $\two$ the second fundamental form of $\psi$, and we write $\dd V$ for the volume form on $(\M,g)$. 
	
	In \cite{la}, motivated by J. Cheeger's finiteness theorems (\cite{ch}, also see K. Corlette \cite{co}) and the Willmore energy of surfaces (see \textit{e.g.} T. Rivi\`{e}re \cite{ri}), J. Langer proved the following result:

	{\em Let $A,E$ are given finite numbers and $p>2$. Denote by $\F(A,E,p)$ the moduli space of immersed surfaces $\psi: \M \map \R^3$ with ${\rm Area}(\psi) \leq A$, $\|\two\|_{L^p(\M)} \leq E$ and $\int_\M \psi\,\dd V =0$. Then any sequence $\{\psi_j\}\subset\F(A,E,p)$ contains a subsequence converging in $C^1$ to an immersed surface modulo  ${\rm Diff}(\M)$, the group of diffeomorphisms of $\M$}.
	
	Here and hereafter, the immersed submanifold $\psi: \M \emb \R^n$ is equipped with the pullback metric from the Euclidean metric on $\R^n$. We denote by $\two$ the second fundamental form of $\psi$, and we write $\dd V$ for the volume form on $\M$. 
	
	 In a recent paper \cite{br}, P. Breuning generalised J. Langer's result to arbitrary dimensions and co-dimensions:
	 
	 {\em Let $A,E$ be given finite numbers and $p>n$, $d>n$. Denote by $\F(V,E,d,n)$ the moduli space of immersions $\psi: \M \map \R^n$ where $\M$ is a $d$-dimensional closed manifold, ${\rm Vol}(\M)\leq A$, $\|\two\|_{L^p(\M)} \leq E$ and $\psi(\M)$ contains a fixed point.  Then any sequence $\{\psi_j\} \subset \F(V,E,d,n)$ contains a subsequence converging in $C^1$ to an immersed submanifold modulo ${\rm Diff}(\M)$.}
	
	The above two compactness theorems on the moduli space of immersions have a crucial assumption: $p > \dim (\M) =d$. Indeed, the proofs in \cite{la, br} utilise the Sobolev--Morrey embedding $W^{1,p}(\R^n)\emb C^{0,\alpha}(\R^n)$, where $p>n$ and $\alpha=\alpha(p,n) \in ]0,1[$. It is natural to ask about the end-point case $p=d$, for which the Sobolev--Morrey embedding fails. In the case $p=d=2$, J. Langer (p.227, \cite{la}) constructed a counterexample using conformal geometry --- the M\"{o}bius inversions of the Clifford torus $T_{\rm cl}$ with respect to a sequence of points $x_j\notin T_{\rm cl}$ approaching an outermost point (with distance measured from the centre of the embedded image of $T_{\rm cl}$)  on $T_{\rm cl}$ cannot tend to any immersed manifold. It crucially relies on the structure of $\C$. 
	
	Our first goal of this paper is to construct a counterexample for $p=d$ in arbitrary dimensions. The idea is to construct a family of  hypersurfaces ``spiralling wildly'', as a vague reminiscence of the motion of vortex sheets in fluid dynamics. This is achieved by letting the Gauss map $\n$ ({\it i.e.}, the outward unit normal vectorfield) take $N \gg 1$ turns in one direction as we approach some fixed point $O$ and,  symmetrically, take $N$ turns in the opposite direction as we leave $O$. To illustrate the geometric picture, we first discuss the case $d=1$, and then construct a counterexample for general $d$. Instead of using conformal geometric methods, we exploit the scaling invariance of $\|\two\|_{L^d(\M)}$, which holds in arbitrary dimensions. This is the content of $\S 2$.

Our second goal is to establish an affirmative compactness result for the $p=d$ case, with the help of an additional hypothesis: the $BMO$-norm of the Gauss map $\n$,
\begin{equation}
\|\n\|_{{\rm BMO}(\M)} := \sup_{x \in \M, \,R>0} \dashint_{\M \cap B(x,R)} |\n(y) - \n_{x,R}|\,\dd V(y),
\end{equation} 
is small. Throughout $B(x,R)$ denotes the geodesic ball of radius $R$ and centre $x$ in $\M$, $\dashint$ is the averaged integral, and $\n_{x,R}:=\dashint_{B(x,R)}\n\,\dd V$. This is inspired by the works \cite{se1, se2, se3} due to S. Semmes on {\em chord-arc surfaces with small constants}. In $\S 3$ we use the results in \cite{se1, se2, se3} to prove a ``partial regularity'' theorem for the weak limit: given any family of immersed hypersurfaces $\R^d$ in the $(d+1)$-dimensional Euclidean space with uniformly $L^d$-bounded second fundamental forms and small $\|\n\|_{{\rm BMO}(\M)}$, one may extract a subsequence whose limit can be locally parametrised by H\"{o}lder functions.

The paper is concluded by several further remarks in $\S 4$.

\section{A counter-example to the endpoint case $p=d$}

Let us first study the toy model $d=1$. We prove the following simple result:
\begin{lemma}\label{lemma: d=1}
There exist a family of smooth  curves $\{\M^\e\}$ each homeomorphic to $\R^1$, and a family of immersions $\psi^\e : \M^\e \map \R^2$ as planar curves, such that the extrinsic curvatures $\{\two^\e\}$ associated to $\{\psi^\e\}$ are uniformly bounded in $L^1$, but $\{\psi^\e \circ \sigma^\e\}$ does not converge in $C^1$-topology to any immersion of $\R$ for arbitrary $\{\sigma^\e\} \subset {\rm Diff}(\R)$. 
\end{lemma}

The extrinsic curvature of a planar curve is the mean curvature. Recall that the mean curvature is defined in arbitrary dimensions as the trace of the second fundamental form. In the case $d=1$ we may still denote the extrinsic curvature by $\two$. %Also, $\{\sigma^\e\}$ are called ``gauges'' in \cite{br}. 

\begin{proof}
	Let $J \in C^\infty_c(\R)$ be a standard symmetric mollifier; {\it e.g.}, 
\begin{equation}
J(s) := \Lambda\exp\bigg\{\frac{1}{s^2-1}\bigg\} \1_{\{|s|<1\}},
\end{equation}	
where the universal constant $\Lambda>0$ is chosen such that $\int_\R J(s)\,\dd s =1$. As usual $J_\e(s):=\e^{-1}J(s/\e)$ for $\e>0$; then $\|J_\e\|_{L^1(\R)}=1$ for every $\e>0$. In addition, define the kernel
\begin{equation}\label{K e}
K_\e(x) := J_\e (x+\e) - J_\e (x-\e).
\end{equation}
It satisfies $\|K_\e\|_{L^1(\R)}=2$, $K_\e \in C^\infty_c(\R)$ and ${\rm spt}(K_\e)=[-2\e,2\e]$; in particular, it is smooth at $0$.

Now, define an angle function
\begin{equation}\label{theta, d=1}
\theta^\e(x) := 10^m\cdot 2\pi\int_{-\infty}^x K_\e(s)\,\dd s,
\end{equation} 
where $m \in \mathbb{Z}_+$ is to be determined. Then, choose the Gauss map $\n^\e \in C^\infty(\R;\mathbb{S}^1)$ by
\begin{equation}\label{n, d=1}
\n^\e(x) := \begin{bmatrix}
\cos \theta^\e(x)\\
\sin \theta^\e(x)
\end{bmatrix}\qquad \text{ for each } x \in \R.
\end{equation}
The extrinsic curvature $\two^\e$ equals to the negative of the gradient of the Gauss map; hence
\begin{align}\label{II for d=1}
|\two^\e(x)| &= \sqrt{\big|\big(-\sin\theta^\e(x)\big)(\theta^\e)'(x)\big|^2 + \big|\big(\cos\theta^\e(x)\big)(\theta^\e)'(x)\big|^2} \nonumber\\
&= |(\theta^\e)'(x)| = (2\pi \cdot 10^m) K_\e(x).
\end{align}
Thus, the $L^1$ norm of  $\{\two^\e\}$ is uniformly bounded by $4\pi \cdot 10^m$.

Let $\psi^\e$ be a smooth immersion that realises the Gauss map $\n^\e$ whose image is the unit circle $\mathbb{S}^1$ in $\R^2$. For each $\eta>0$, we may easily modify $\psi^\e$ to $\tilde\psi^\e$ such that $|\tilde\psi^\e(x)|$ is decreasing on $]-\infty, 0]$ and increasing on $[0,\infty[$, the image of $\tilde\psi^\e$ in $\R^2$ is homeomorphic to $\R^1$, and that
\begin{equation}\label{C 100}
\|\psi^\e-\tilde\psi^\e\|_{C^{100}(\R)} < \eta.
\end{equation}
Indeed, notice that the image of $\psi^\e\big|{]-\infty,0]}$ covers $\mathbb{S}^1$ for $10^m$ times in the positive orientation, and the image of $\psi^\e\big|{[0,\infty[}$ covers $\mathbb{S}^1$ for $10^m$ times in the negative orientation. We then choose the perturbed map $\tilde\psi^\e$ such that 
\begin{itemize}
\item
As $x$ goes from $-\infty$ to $0$, $\tilde\psi^\e$ wraps around the origin in a helical trajectory for $10^m$ times. Moreover, in each round $|\tilde\psi^\e|$ decreases monotonically by  $\sim 10^{-m}$;
\item
As $x$ increases from $0$ to $\infty$, $\tilde\psi^\e$ ``unwraps'' around the origin along a helix for $10^m$ times, in each round $|\tilde\psi^\e|$ increases monotonically by $\sim 10^{-m}$;
\item
For $x \in ]-\infty, -2\e] \sqcup [2\e, + \infty[$, the image of $\tilde{\psi}^\e$ consists of straight line segments (``long flat tails''); hence $\n^\e$ stays constant on each component of $]-\infty, -2\e] \sqcup [2\e, + \infty[$;
\item
Finally,  the image $\tilde\psi^\e(\R)$ is $C^\infty$ and homeomorphic to $\R^1$.
\end{itemize}
In view of the above properties, one can take $m=m(\eta)\in\mathbb{Z}_+$ sufficiently large to verify \eqref{C 100}. Let us pick $\eta = \frac{1}{100}$, so $m$ is a universal constant fixed once and for all. Without loss of generality, from now on we may assume $\psi^\e = \tilde\psi^\e$. The point is to ensure that the image of $\psi^\e$ in $\R^2$ is free of loops and ``concentrates'' near the origin $0 \in \R^2$, with Gauss map and second fundamental form arbitrarily staying close to those constructed in \eqref{n, d=1} and \eqref{II for d=1}, respectively.

To conclude the proof, let us define $\M^\e$ as the  homeomorphic copy of $\R^1$ equipped with the pullback metric $(\psi^\e)^\# \delta_{ij}$, where $\delta_{ij}$ is the Euclidean metric on the ambient space $\R^2$. It remains to show that the $C^1$-limit (modulo ${\rm Diff}(\R^1)$) of $\psi^\e$ as $\e\map 0^+$ cannot be an immersion. Indeed, note that the topological degree satisfies
\begin{equation}\label{deg}
\deg \Big(\psi^\e\big|]-\infty,0]\Big) = 10^m,\qquad \deg \Big(\psi^\e\big|[0,\infty[\Big) = -10^m.
\end{equation}
These identities are independent of $\e$. Hence, if $\bar{\psi}$ were a limiting immersion, \eqref{deg} would have been preserved. However, $K_\e \weakstar \delta_0 - \delta_0 =0$ as measures, so \eqref{theta, d=1}\eqref{n, d=1}\eqref{II for d=1} imply that any pointwise subsequential limit of $\psi^\e$ has zero topological degree. Hence we get the contradiction and the proof is complete.   \end{proof}

Three remarks are in order:

{\bf 1.} From \eqref{II for d=1} one may infer that
\begin{equation*}
\|\two^\e\|_{L^\infty(\M^\e)} = \frac{2\pi \cdot 10^m \cdot \Lambda}{e\e} + \eta \longrightarrow \infty \qquad \text{ as } \e \map 0^+.
\end{equation*}

{\bf 2.} The construction in Lemma \ref{lemma: d=1} can be localised near $0$. We can restrict $\M^\e$ to curves of finite $\mathcal{H}^1$ measure by removing the long tails. This recovers the volume bounds in \cite{la,br} ($\S 1$). 

{\bf 3.} We can construct $\phi^\e$ whose limit blows up at a countable discrete set $\{x_n\}$ by taking
$$
\tilde\theta^\e (x) := \sum_{n=1}^\infty 2^{-n} \1_{B(x_n, R_n)} (x) \theta^\e(x)
$$
in place of $\theta^\e(x)$, where $\{B(x_n, R_n)\}$ are disjoint for all $n$. Geometrically, the immersed images corresponding to $\tilde{\theta}^\e$ are smooth curves that spiral towards the centres $x_n$ when $x<x_n$, and then spiral away from $x_n$ when $x>x_n$. Near $x_n$ the rate of motion blows up in $L^\infty$ as $\e\map 0^+$; nevertheless, its $L^1$ norm is constant.

Now let us generalise the above construction to $d$-dimensions:
\begin{theorem}\label{thm: d dimensions}
Let $d \geq 1$ be an integer. There exist a family of smooth manifolds $\{\M^\e\}$ each homeomorphic to $\R^d$, and a family of immersions $\psi^\e : \M^\e \map \R^{d+1}$ as smooth hypersurfaces, such that the second fundamental forms $\{\two^\e\}$ associated to $\{\psi^\e\}$ are uniformly bounded in $L^d$, but $\{\psi^\e \circ \sigma^\e\}$ does not converge in $C^1$-topology to any immersion of $\R^d$ for arbitrary $\{\sigma^\e\} \subset {\rm Diff}(\R^d)$. 
\end{theorem}

\begin{proof}
	Again the crucial point is to construct the Gauss map $\n^\e \in C^\infty(\R^d; \mathbb{S}^d)$. We make use of the spherical coordinates on $\sd$. For $x=(x_1,x_2,\ldots, x_d) \in \R^d$, one needs to specify the angle functions $\theta^\e_i: \R^d \map [0,2\pi[$ in the following:
	\begin{equation}\label{spherical coord}
	\n^\e(x) = \begin{bmatrix}
	\cos \theta_1^\e (x)\\
	\sin\theta_1^\e(x) \cos\theta_2^\e(x)\\
	\sin\theta_1^\e(x) \sin\theta_2^\e(x)\cos\theta_3^\e(x)\\
	\vdots\\
	\sin\theta_1^\e(x)\cdots\sin\theta_{d-1}^\e(x)\cos\theta_d^\e(x)\\
	\sin\theta_1^\e(x)\cdots\sin\theta_{d-1}^\e(x)\sin\theta_d^\e(x)
		\end{bmatrix}.
	\end{equation}
Throughout $\sd = \{z \in \R^{d+1}: |z|=1\}$ is the round sphere. 

Indeed, let us choose 
\begin{equation}\label{theta i def}
\theta^\e_i (x) \equiv \Theta^\e(x_i) := 10^m \cdot 2\pi \int_{-\infty}^{x_i} K_\e(s)\,\dd s,
\end{equation}
where the kernel $K_\e$ is defined as in \eqref{K e}, and $m \in \mathbb{Z}_+$ is a large universal constant to be fixed later. Each $\theta^\e_i$ is a function of $x_i$ only. One can easily compute all the entries in $-\two^\e = \na \n^\e$, which is a lower-triangular $d \times (d+1)$ matrix due to the embedding $\sd \emb \R^{d+1}$. The rows $\{\mathbf{r}_i\}_{i=1,2,\ldots, d}$ of $\{\na\n^\e\}$ are:
\begin{equation*}
\mathbf{r}_1 = \Big( -(\Theta^\e)'(x_1) \sin \Theta^\e(x_1), 0 , \cdots, 0 \Big),
\end{equation*}
\begin{equation*}
 \mathbf{r}_2 = \Big((\Theta^\e)'(x_1) \cos\Theta^\e(x_1) \cos \Theta^\e(x_2), -(\Theta^\e)'(x_2) \sin\Theta^\e(x_1) \sin\Theta^\e(x_2) ,0,\cdots,0\Big),
\end{equation*}
\begin{align*}
&\mathbf{r}_3 = \Big( (\Theta^\e)'(x_1) \sin\Theta^\e(x_1) \sin\Theta^\e(x_2) \cos\Theta^\e(x_3), (\te)'(x_2)\sin\te(x_1)\cos\te(x_2)\cos\te(x_3),\\
&\qquad\qquad -(\te)'(x_3) \sin\te(x_1)\sin\te(x_2)\sin\te(x_3), 0, \cdots, 0 \Big)
\end{align*}
so on and so forth, with the last two being
\begin{align*}
&\mathbf{r}_{d-1} = \Big( (\te)'(x_1)\cos\te(x_1)\sin\te(x_2)\cdots\sin\te(x_{d-1})\cos\te(x_d), \star, \cdots, \star,\nonumber\\
&\qquad\qquad (\te)'(x_{d-1})\sin\te(x_1)\sin\te(x_2)\cdots\cos\te(x_{d-1})\cos\te(x_d),\nonumber\\
&\qquad\qquad -(\te)'(x_d)\sin\te(x_1)\sin\te(x_2)\cdots\sin\te(x_{d-1})\sin\te(x_d)\Big)
\end{align*}
and
\begin{align*}
&\mathbf{r}_{d} =  \Big((\te)'(x_1)\cos\te(x_1)\sin\te(x_2)\cdots\sin\te(x_{d-1})\sin\te(x_d),\star, \cdots, \star,\\
&\qquad\qquad  (\te)'(x_{d-1}) \sin\te(x_1)\sin\te(x_2)\cdots\cos\te(x_{d-1})\sin\te(x_d),\\
&\qquad\qquad(\te)'(x_d)\sin\te(x_1)\sin\te(x_2)\cdots\sin\te(x_{d-1})\cos\te(x_d) \Big).
\end{align*}

A direct computation yields the Hilbert--Schmidt norm of the second fundamental form:
\begin{equation}
|\two^\e| = |\na\n^\e| = \Big|\Big((\te)'(x_1), \cdots, (\te)'(x_d)\Big)\Big|.
\end{equation}
Thus, in view of \eqref{theta i def} and Fubini's theorem, we have
\begin{equation}\label{two, d dim}
\|\two^\e\|_{L^d(\R^d)} = 10^m \cdot 2\pi\,\Big\|\underbrace{ K_\e \otimes \cdots\otimes K_\e}_{d \, \text{ times } } \Big\|_{L^d(\R^d)} = 10^m \cdot 2\pi\,\|K_\e\|_{L^1(\R^d)} = 10^m \cdot 4\pi.
\end{equation}

It now remains to choose a smooth immersion that realises   $\n^\e$ (approximately). The construction is similar to Lemma \ref{lemma: d=1} in the case of  $d=1$. First, take $\psi^\e$ whose Gauss map is $\n^\e$. Geometrically, $\psi^\e$ winds around $\sd$ --- accelerating on the first half and decelerating on the second half of the trajectory --- with respect to a given orientation for $10^m$ ``cycles'', and then undoes the winding by turning symmetrically in the opposite orientation. In the above, by ``cycle'' we mean a generator of the cohomology group $H^d(\sd) \cong \R$.

In what follows we shall describe how to modify the above construction to obtain a counterexample with $\M^\e$ homeomorphic to $\R^d$. We shall construct $\tilde{\psi}^\e$, modified versions of $\psi^\e$, such that for each $\e>0$ the image of $\tilde{\psi}^\e$ in $\R^{d+1}$ is a smooth homeomorphic copy of $\R^d$. In addition, each such image has flat ends outside $B(0,2)$ and has $d$ independent angular variables in the spherical coordinates, {\it i.e.}, the Gauss map still takes the form \eqref{spherical coord} with $\theta_i^\e(x) \equiv \Theta^\e(x_i)$. Moreover, such $\tilde\psi^\e$ differs from $\psi^\e$ only by an arbitrarily small error in the $C^{100}$-topology.

To obtain the modified immersions $\tilde{\psi}^\e$, let us begin with $\mathcal{S}^\e \equiv $ image of $\psi^\e$. For each $\e>0$, schematically we can write $\mathcal{S}^\e \sim +10^m \sd- 10^m \sd$, with $\pm$ denoting the orientation. For $x \notin B(0,2)$ we have $\psi^\e \equiv 0$, in view of \eqref{theta i def} and the choice of $K_\e$. Now, for some small number $0<\eta \ll 10^{-m}$ to be specified, we shall modify $\mathcal{S}^\e$ as follows.

First of all, for the concentric spheres $\sd$ and $(1-\eta)\sd:=\p B(0, 1-\eta)$ in $\R^{d+1}$, we can smoothly ``interpolate'' between them by finding a hypersurface $\mathcal{S}_1^\e$ lying in the annulus formed by the two spheres, such that the tangent spaces of $\mathcal{S}_1^\e$ and $\mathcal{S}^\e$ coincide at the ``north poles'' $\mathbf{e}$ and $(1-\eta)\mathbf{e}$, and that all the angular variables $\theta_1,\ldots\theta_d$ on $\mathcal{S}_1^\e$ change by $2\pi$ at the same constant speed. Here and hereafter $\mathbf{e}:=(0,\ldots,0,1)$. In the same way, we construct $\mathcal{S}^\e_2$ nested between $(1-\eta)\sd$ and $(1-2\eta)\sd$, so that  $\mathcal{S}^\e_1$ and  $\mathcal{S}^\e_2$ can be glued smoothly at $(1-\eta)\mathbf{e}$, and their natural orientations are the same. Let us repeat this process to get $S^\e_j$ for $j\in\{1,2,\ldots, 10^m\}$, and glue $$\bigcup_{j=1}^{10^m}\mathcal{S}^\e_j :=\mathcal{S}^\e_+$$ to form a smooth ``spiral'' starting from $\mathbf{e}$ and ending at $(1-10^m\eta)\mathbf{e}$. 

To proceed, denote by $\mathcal{S}^\e_-$ the hypersurface obtained via shifting $\mathcal{S}^\e_+$ to its right-hand side by $\eta\slash 2$. This is well-defined as $\mathcal{S}^\e_+$ an oriented hypersurface in $\R^{d+1}$, and we have $\mathcal{S}^\e_+ \cap \mathcal{S}^\e_- = \emptyset$. Let us endow $\mathcal{S}^\e_-$ with the  orientation opposite to that of $\mathcal{S}^\e_+$. Furthermore, we may find a short neck $\G^\e$, such that $\G^\e$ is a smooth hypersurface disjoint with $\mathcal{S}^\e_+$, $\mathcal{S}^\e_-$, and that $\mathcal{S}^\e_+ \cup \G^\e \cup \mathcal{S}^\e_-$ can be glued together smoothly. Additionally let us require that on $\G^\e$  each of the angular variables $\theta_1, \ldots,\theta_d$ does not vary more than $\pi\slash 1000$. Also, the area of $\G^\e$ is entailed to shrink to zero as $\e \map 0^+$. Finally, at the points $\mathbf{e}$ and $(1+\eta\slash 2)\mathbf{e}$, we glue to $\mathcal{S}^\e_+ \cup \G^\e \cup \mathcal{S}^\e_-$ the Euclidean half-planes $\Pi_+$ and $\Pi_-$, respectively, such that $\Pi_\pm$ are isomorphic copies  $\R^d_+$ and that
\begin{equation}\label{S glue}
\mathcal{S}^\e := \Pi_+ \cup S^\e_+ \cup \G^\e \cup \mathcal{S}^\e_- \cup \Pi_- 
\end{equation}
is a smooth hypersurface homeomorphic to $\R^{d} \subset \R^{d+1}$.

We conclude the construction by setting $$\M^\e := \big(\mathcal{S}^\e, (\tilde\psi^\e)^\# \delta_{ij}\big),$$ where $\delta_{ij}$ is the Euclidean metric on $\R^{d+1}$. Note that for each $\e>0$ the manifold $\M^\e$ is homeomorphic to $\R^d$. The origin $0 \in \R^d$ corresponds to the point that lies in the neck $\G^\e$ for all $\e>0$. One may think of each of the variables $x_i$, $i\in\{1,2,\ldots,d\}$ in \eqref{spherical coord} and the ensuing arguments as the ``time'' variable analogous to $x$ in Lemma \ref{lemma: d=1}, and view $x=(x_1,\ldots,x_d)$ as being restricted to the diagonal. Thus, each angular variable $\theta^\e_i$ behaves in the same way as $\theta^\e$ in Lemma \ref{lemma: d=1}, and these variables are ``synchronised''.

In the above construction of $\M^\e$, we find that the Gauss map $\n^\e$ of the immersion (in fact, embedding) $\tilde\psi^\e$ still takes the form of \eqref{spherical coord} with $\theta_i^\e(x) \equiv \tilde\te(x_i)$. Moreover, each step of the construction can be performed with sufficiently small perturbations in any norm, say $C^{100}$; thus 
\begin{equation}\label{close, C100}
\|\psi^\e - \tilde\psi^\e\|_{C^{100}} \leq C\eta
\end{equation}
for a universal constant $C$. So the second fundamental forms $\tilde\two^\e$ for $\tilde{\psi}^\e$ are also uniformly close to $\two^\e$, say in the $C^{97}$-topology. In particular, as $\tilde\two^\e$ and $\two^\e$ are both compactly supported on $B(0,2)$, we deduce from \eqref{two, d dim} that $$\Big|\|\tilde{\two}^\e\|_{L^d(\R^d)} - 10^m\cdot 4\pi\Big| \leq C\eta,$$
where $C$ is a universal constant. Hence $\|\tilde{\two}^\e\|_{L^d(\R^d)}$ is uniformly bounded.

Finally, let us consider the  topological degree for $\tilde\psi^\e$. By construction, $\Pi_+, \G^\e$ and $\Pi_-$ do not contribute to the degree. Also, the following holds independently of $\e$:
\begin{equation}\label{deg, nd}
\deg \big(\tilde{\psi}^\e | \mathcal{S}^\e_\pm\big) = \pm 10^m.
\end{equation}
Indeed, thanks to the definition of $\psi^\e$ and \eqref{close, C100}, the images under 
$\tilde{\psi}^\e$ of $\mathcal{S}^\e_\pm$ are $10^m$ times the non-trivial generator of $H^d(\sd)$ with the positive and negative orientations, respectively. Now suppose $\bar{\psi}$ were a limit $\{\tilde{\psi}^\e\}$ as an immersion of hypersurface, then \eqref{deg, nd} would have been preserved. Whereas, by \eqref{close, C100}\eqref{theta i def} and in light of the construction of the kernel $K_\e$ and the neck $\G^\e$, any pointwise subsequential limit of $\tilde\psi^\e$ has zero degree. So, in light of the diffeomorphism-invariant property of the topological degree, $\bar{\psi}$ cannot be an immersion modulo the action of ${\rm Diff}(\R^n)$. This completes the proof.  \end{proof}

Similar to the remarks ensuing the proof of Lemma \ref{lemma: d=1}, this counterexample can be localised, and we can get a family of immersions of $\R^d$  blowing up at an infinite discrete set.

\section{Local H\"{o}lder Regularity}
%In this section we deduce a compactness theorem utilising the works \cite{se1, se2, se3} of S. Semmes on the harmonic analysis on chord-arc surfaces with small constants.

Consider the moduli space
\begin{align}
\F(\delta, d) &:= \bigg\{ f \in W^{2,d}\cap C^\infty(\M; \R^{d+1}) : \text{$f$ is an immersion}, \text{ $\M$ is an $d$-dimensional hypersurface}, \nonumber\\
&\qquad \text{$\M\cup\{\infty\}$ is smooth in $\mathbb{S}^{d+1}$}, \|\n\|_{{\rm BMO}(\M)} \leq \delta, \, \text{$f(\M)$ contains a fixed point} \bigg\}.
\end{align}
Heuristically, we show the following: if the Gauss maps of a family of smooth  homeomorphic $\R^d$ have uniformly small oscillations at all scales, then ``a little''  regularity persists in the limit. %This assumption is natural. In fact, if a family of $W^{2,d}$ immersions of $d$-manifolds has uniformly $L^d$-bounded second fundamental forms, then their Gauss maps have bounded $BMO$-norms (provided that Poincar\'{e} and Sobolev inequalities hold). 

To state the result rigorously, we need the following
\begin{definition}
A set $\Omega \subset \R^d$ is a H\"{o}lder graph system if it can be locally represented by graphs of $C^{0,\gamma}$-functions for some $\gamma \in ]0,1]$.
\end{definition}
The notion of ``graph systems'' plays an essential role in the works \cite{la, br}. Note that we do not require further geometric information for a H\"{o}lder graph system, {\it e.g.}, whether or not it represents a topological manifold or orbifold. 

Our main result of this section can now be stated as follows:
\begin{theorem}\label{thm: holder}
There exists a small constant $\delta_0>0$ depending only on the dimension $d$, such that for any $\delta \in [0,\delta_0]$ and any family of immersions $\{\psi^\e\} \subset \F(d,\delta)$, we can find $\{\sigma^\e\}\subset {\rm Diff}(\R^d)$ such that, after passing to subsequences,  $\{\psi^\e \circ \sigma^\e\}$ converges to a H\"{o}lder graph system.
\end{theorem}

The proof is based on the framework and results developed by S. Semmes (\cite{se1, se2, se3}) on the harmonic analysis on chord-arc surfaces with small constants:

\begin{definition}[See Main Theorem, p.200 of \cite{se1}]\label{def: chord-arc}
Let $\M$ be a hypersurface in $\R^{d+1}$. It is a chord-arc surface with small constant if $\gamma>0$ is small, or equivalently, $\eta>0$ is small.  $\gamma$ is defined to be the smallest number such that 
\begin{enumerate}
\item[$(\gamma 1)$]
The $BMO$-norm of the Gauss map $\n$ is no larger than $\gamma$;
\item[$(\gamma 2)$]
For each $x\in\M$, $R>0$ and $y \in B(x,R)$, there holds $|(x-y) \cdot \n_{x,R}| \leq \gamma R$.
\end{enumerate}
Here and hereafter $f_{x,R} := \dashint_{B(x,R)}f = {\rm Vol}^{-1}[B(x,R)] \int_{B(x,R)}f$ for each function $f$. On the other hand, $\eta>0$ is the smallest number such that 
\begin{enumerate}
\item[$(\eta 1)$]
For every $x\in\M$, $R>0$, 
\begin{equation*}
\bigg| \frac{{\rm Vol}\,\big(\M \cap B(x,R)\big)}{{\rm Vol}\,\big(B(0,1)\big) R^d} - 1\bigg| \leq \eta;
\end{equation*}
\item[$(\eta 2)$]
For any $x, y\in\M$, $d(x,y)\leq (1+\eta)|x-y|$. Here $d$ denotes the geodesic distance on $\M$. 
\end{enumerate}
\end{definition}

In fact, in the Main Theorem, p.200 of \cite{se1}, the equivalence of the two conditions in Definition \ref{def: chord-arc} is proved, together with yet another two equivalent conditions $(\alpha)$ and $(\beta)$ defined via Clifford--Cauchy integrals and Hardy spaces. Each of the numbers $\alpha,\beta,\gamma,\eta$ can be called ``the chord-arc constant''.

Our proof of Theorem \ref{thm: holder} relies crucially on three results in \cite{se1, se2, se3}. Let us first discuss these results instinctively and non-rigorously by emphasising their geometric meanings. The exact statements, quotations and explicit estimates will be presented in the proof. 

\begin{itemize}
\item
For $\M$ in the moduli space $\mathcal{F}(\delta,d)$, if $\delta$ is sufficiently small, then it is proved in \cite{se2, se3} that $\M$ is a chord-arc surface with small constant. In other words, when restricted $\mathcal{F}(\delta,d)$, $(\gamma 1)$ implies $(\gamma 2)$. In particular, for each $x\in \M$ and $R>0$, $B(x,R)\cap\M$ stays close to the hyperplane through $x$ normal to the averaged Gauss map $\n_{x,R}$. 
\item
A chord-arc surface $\M$ with small constant can be ``smoothed'' in a small neighbourhood around any $x \in \M$: there exists another chord-arc surface which is a Lipschitz graph, and which stays very close to $\M$. 
\item
A chord-arc surface $\M$ with small constant has a ``bi-H\"{o}lder'' parametrisation by $\R^d$.\footnote{It remains an open question if we change ``bi-H\"{o}lder'' to ``bi-Lipschitz''; see \cite{to} by T. Toro for discussions.}
\end{itemize}

One more issue before presenting the proof: we need
\begin{definition}
Let $(X,d)$ be a metric space. $\mathcal{N} \subset X$ is said to be a $R$-net if $X = \bigcup_{z \in \mathcal{N}} B(z,R)$.  Moreover, $\tilde{\mathcal{N}}$ is said to be a $\tilde{R}$-subnet of $\mathcal{N}$ if $X = \bigcup_{\tilde{z} \in \tilde{\mathcal{N}}} B(\tilde{z}, \tilde{R})$,  and if for each $\tilde{z} \in \tilde{\mathcal{N}}$ one can find some $z \in \mathcal{N}$ such that  $B(\tilde{z}, \tilde{R}) \subset B(z,R)$. 
\end{definition}
In the above $B(\bullet,\bullet)$ are the metric balls, and by an abuse of notations we also refer to $\{B(z,R):z\in\mathcal{N}\}$ as the $R$-net.

\begin{proof}[Proof of Theorem \ref{thm: holder}]
Assume $\M \in \mathcal{F}(\delta,d)$ with $\delta \leq \delta_0$ to be chosen. Fix any $t>0$, {\it e.g.} $t=10^{-5}$. By $\S 3$, \cite{se2} one can find another chord-arc surface $\M_t$ with the constant $\mu$ to be specified, such that $$0\leq\delta \leq \delta_0 \leq C(d)\delta_0 <\mu.$$

%We shall choose $\mu$ later, which is equivalent to the least upper bound for the $BMO$-norm of the Gauss map; see p.200 \cite{se1}.
 Next, in view of Eq.\,(3.7) and Lemma 3.8 in \cite{se2}, $$\M_t \cap B\big(x, (2^{-1}+10^{-10})t\big)$$ is a Lipschitz graph with constant $\leq C_0\mu$ for each $x\in \M$, provided that $\mu=\mu(t,\delta_0)$ is chosen large enough. Here $C_0=C(d,\delta_0)$. Under the same condition, $\M_t$ can be taken sufficiently close to $\M$. More precisely, by  Lemma 3.8 in \cite{se2}, one may take
 \begin{equation*}
 {\rm dist}(\M_t, \M) \leq 10^{-10}t.
 \end{equation*}

Moreover, by Theorem 4.1 in \cite{se2}, there exists a homeomorphism $\tau: \M \map \M_t$ such that
	\begin{equation}\label{bi-holder estimate}
	\max\Big\{\|\tau\|_{C^{0,\gamma}(B(x,100 t)\cap\M)}, \|\tau^{-1}\|_{C^{0,\gamma}(B(x,100 t))\cap\M_t} \Big\}	\leq C_1 \qquad \text{ for all }x \in \M,
	\end{equation}
where $C_1=C(d, \delta_0, t)$ and the H\"{o}lder index is given by
\begin{equation}\label{gamma}
\gamma \equiv 1-C_2d\delta_0
\end{equation}
for a dimensional constant $C_2$ (denoted by $k$ in \cite{se2}). In fact, putting together Eqs.\,(1.3)(4.6)\footnote{Eq.\,(4,6) in \cite{se2} contains an index  $p$; for our purpose we can take it to be $(C_2\gamma)^{-1}$, by Theorem 4.1 in \cite{se2}}, Lemma 5.5 in \cite{se2} and that $0 \leq \delta \leq \delta_0$, we may explicitly select
\begin{equation}\label{C1}
C_1 = C_3^{C_2\delta_0} \bigg\{\frac{(100 t)^{C_2\delta_0}}{1-2 \cdot 10^d \delta_0}\bigg\}.
\end{equation}
Here $C_3=C_3(d)$ is a  dimensional constant. Notice that our estimates \eqref{C1}\eqref{bi-holder estimate} are uniform in $\delta$. We restrict to $\delta_0<(C_2d)^{-1}$ to ensure that $\gamma >0$ in \eqref{gamma}.

With the above explicit estimates at hand, we are ready to conclude the theorem. By considering a compact exhaustion $\{\M_k\} \nearrow \M$, one may take $\M$ to be a bounded domain in $\R^d$. %(The argument for non-compact manifolds in the $p>d$ case is more involved, if one needs to check that the limiting object is a manifold; see $\S 7$ in \cite{br}.) 
Then we can take a $(50 t)$-net $\mathcal{N}$ of $\M$, whose cardinality is
$$
\mathcal{H}^0(\mathcal{N}) = C_4 t^{-d}
$$
for some geometric constant $C_4=C(d,\gamma)\equiv C(d,\delta_0)$. Restricted to each member in the $\mathcal{N}$, the hypersurface  $\M$ is $C^{0,\gamma}$-parametrised by $\M_t$,  a Lipschitz graph over $(2^{-1}+10^{-10})$-balls. Using the quantitative estimates in the preceding paragraphs, we can refine $\mathcal{N}$ to a subnet $\tilde{\mathcal{N}}$ with cardinality $C_5 t^{-d}$, where $C_5=C(d,\delta_0)$ again, such that in each $B \in \tilde{\mathcal{N}}$ the set $B \cap \M$ is parametrised by a $C^{0,\gamma}$-homeomorphism with the H\"{o}lder norm bounded by $C_6:=C_0\mu\cdot C_1$.

To complete the proof, let us choose $$\mu = 10 C(d)\delta_0.$$ By carefully tracing the dependence of the constants $C_1,\ldots, C_5$ in the above arguments, one concludes that $C_6=C(d,\delta_0,t)$. But $t=10^{-5}$ is fixed from the beginning of the proof, so $C_6$ depends only on the dimension $d$ and $\delta_0$, the upper bound for the $BMO$-norm of the Gauss map. Therefore, the assertion now follows from the Arzela--Ascoli theorem, {\it i.e.}, the compactness of $C^{0,\gamma} \emb C^{0,\gamma'}$ for $\gamma' \in ]0,\gamma[$. \end{proof}

\section{Three Further Questions}

{\bf 1.} 
Let the moduli space $\F(A,E,p)$ be as in $\S 1$. Is the subspace
\begin{equation*}
\F_{\rm isom}(A,E,p) := \Big\{\psi \in \F(A,E,p): \psi \text{ is an isometric immersion of a fixed manifold } \M\Big\}
\end{equation*}
compact in its natural topology? For the end-point case $p=2=d$ the answer is affirmative, in contrast to the unconstrained case for $\F(A,E,p)$. The authors of \cite{cl} proved this via establishing the weak continuity of the Gauss--Codazzi equations (the PDE system for the isometric immersion), with the help of a div-curl type lemma due to Conti--Dolzmann--M\"{u}ller in  \cite{cdm}. What about higher dimensions $d \geq 3$ (and co-dimensions greater than $1$)? That is, for a family of isometric immersions of some fixed $d$-dimensional manifold with uniformly bounded second fundamental forms in $L^d$, is the subsequential limit an isometric immersion?

\smallskip
{\bf 2.} Theorem \ref{thm: holder} leaves open the possibility that the limiting objects of $W^{2,d}$-bounded immersed hypersurfaces may be very irregular ({\it e.g.}, the nowhere differentiable Weierstrass function is $C^{0,\gamma}$, or other fractals), even if the  geometrical condition that the Gauss map is slowly oscillating is enforced. Can we find natural geometrical conditions on the moduli space of $d$-dimensional hypersurfaces with uniformly bounded second fundamental forms in $L^d$, which is sufficient to ensure higher regularities for the subsequential limits, {\it e.g.}, $BV$ or Lipschitz? This is related to the problem of finding good parametrisations of chord-arc surfaces; see the discussions by S. Semmes \cite{se2} and T. Toro \cite{to}.

\smallskip
{\bf 3.} Theorem \ref{thm: d dimensions} shows that space of smooth hypersurfaces in $\R^{d+1}$ with uniformly $L^d$-bounded second fundamental forms is non-compact modulo diffeomorphisms. Under what additional conditions can we retain compactness? For the simplest case, under what extra geometrical or analytical assumptions is the space of topological $\mathbb{S}^2$ immersed in $\R^3$ with $\|\two\|_{L^2} \leq E$ compact?

\bigskip
\noindent
{\bf Acknowledgement}.
This work has been done during the author's stay as a CRM--ISM postdoctoral fellow at Centre de Recherches Math\'{e}matiques, Universit\'{e} de Montr\'{e}al and Institut des Sciences Math\'{e}matiques. Siran Li would like to thank these institutions for their hospitality. The author is also indebted to Prof.\,Gui-Qiang G. Chen and Prof.\,Pengfei Guan for their continuous support and many insightful discussions on isometric immersions.

\end{document}